\documentclass[10pt]{article}
\usepackage{amssymb}
\usepackage{amsmath}
\usepackage{amsfonts}
\usepackage{amsthm}
\usepackage{graphics,graphicx}
\usepackage{float}
\usepackage{mathtools}
\usepackage[dvipsnames]{xcolor}
\usepackage{tkz-graph}
\usetikzlibrary{shapes, arrows.meta, positioning}

\newtheorem{theorem}{Theorem}
\newtheorem{lemma}{Lemma}
\newtheorem{proposition}{Proposition}
\newtheorem{corollary}{Corollary}
\newtheorem{example}{Example}

\newtheorem{remark}{Remark}
\newtheorem{problem}{Problem}
\newtheorem{o-problem}{Open problem}
\renewenvironment{proof}{{\itshape Proof.}\noindent}{\qed}
\makeatletter
\renewcommand{\thefigure}{\@arabic\c@figure}
\makeatletter

\usepackage{authblk}

\title{Bicrucial $k$-power-free permutations}
\author[1]{Margarita Akhmejanova\thanks{mechmathrita@gmail.com}}
\author[2]{Aiya Kuchukova\thanks{akuchukova3@gatech.edu}}
\author[3]{Alexandr Valyuzhenich \thanks{corresponding author; graphkiper@mail.ru}}
\author[4]{Ilya Vorobyev \thanks{vorobyev.i.v@yandex.ru}}

\affil[2]{Georgia Institute of Technology, Atlanta, US}
\affil[3]{Postdoctoral Research Station of Mathematics, Hebei Normal University, Shijiazhuang 050024, China}
\affil[4]{IOTA Foundation, Berlin, Germany}
\date{}
\begin{document}
  \maketitle
  
\begin{abstract}
In this work, we prove that for every $k\geq 3$ there exist arbitrarily long bicrucial $k$-power-free permutations.
We also show that for every $k\geq 3$ there exist right-crucial $k$-power-free permutations of any length at least $(k-1)(2k+1)$.
\end{abstract}  

\textbf{Keywords:} bicrucial, right-crucial, crucial, square-free

\textbf{AMS classification:} 68R15, 05A05

\section{Introduction}
Bicrucial (maximal) abelian square-free words were investigated in \cite{B04,CM01,K03}, and
right-crucial (crucial) abelian $k$-power-free words were investigated in \cite{AGHK10,GHK10}.
In 2011, Avgustinovich et al. \cite{AKPV11}  initiated the study of bicrucial and right-crucial square-free permutations.
In particular, they proved that there exist bicrucial square-free permutations of lengths $8m+1$, $8m+5$, $8m+7$ for all $m\geq 1$.
In 2015, Gent et al. \cite{GKKLN15} showed that bicrucial square-free permutations of even length exist, and the smallest
such permutations are of length $32$.
They also showed that bicrucial square-free permutations of length $8m+3$ exist for $m=2,3$ and they do not exist for $m=1$.
In 2022, Groenland and Johnston \cite{GJ22} completed the classification of $n$ for which there
exist bicrucial square-free permutations of length $n$. Groenland and Johnston posed the following problem (see Problem 5.1 in \cite{GJ22}).

\begin{problem}
 Are there arbitrarily long bicrucial $k$-power-free permutations for $k\geq 3$? 
\end{problem}

In this work, we prove that for every $k\geq 3$ there exist arbitrarily long bicrucial $k$-power-free permutations.
We also show that for every $k\geq 3$ there exist right-crucial $k$-power-free permutations of any length at least $(k-1)(2k+1)$.

The paper is organized as follows. In Section \ref{Sec:Definitions}, we introduce basic definitions. 
In Section \ref{Sec:Prelim}, we give preliminary results.
In Section \ref{Sec:ConstrCrucial}, we give constructions of right-crucial $k$-power-free permutations.
In Section \ref{Sec:ConstrBicrucial}, we give constructions of bicrucial $k$-power-free permutations.
In Section \ref{Sec:Crucial}, we prove that for every $k\geq 3$ there exist right-crucial $k$-power-free permutations of any length at least $(k-1)(2k+1)$.
In Section \ref{Sec:Bicrucial}, we prove that for every $k\geq 3$ there exist arbitrarily long bicrucial $k$-power-free permutations.
In Section \ref{Open problem}, we state an open problem.

\section{Basic definitions}\label{Sec:Definitions}

A {\em permutation} of length $n$ is a word of length $n$ whose symbols are distinct integers.
For example, $213$ and $574$ are permutations of length $3$.
The length of a permutation $P$ is denoted by $|P|$.
A permutation $Q$ is a {\em factor} of a permutation $P$ if $P=P_1QP_2$ for some permutations $P_1$ and $P_2$ ($P_1$ and $P_2$ may be empty).
For example, $754$ and $543$ are factors of the permutation $275436$.
Note that any factor of a permutation is also a permutation.
A permutation $Q$ is a {\em suffix} of a permutation $P$ if $P=P_1Q$ for some permutation $P_1$.
A permutation $Q$ is a {\em prefix} of a permutation $P$ if $P=QP_2$ for some permutation $P_2$.

For a positive integer $n$, denote $[n]=\{1,\ldots,n\}$.
Two permutations $p_1\ldots p_n$ and $q_1\ldots q_n$ are called {\em order-isomorphic} 
if $(p_i-p_j)(q_i-q_j)>0$ for all $i,j\in [n]$, $i\neq j$.
For example, $1592$ is order-isomorphic to $1342$. 
We denote this equivalence by $\sim$.
For a positive integer $k\geq 2$, a {\em $k$-power} is a permutation of the form $X_1X_2\ldots X_k$, 
where $X_i\sim X_j$ for all $i,j\in [k]$ and $|X_1|\geq 2$.
For example, the permutation $461523$ is a $3$-power.
We say that a permutation $P$ {\em contains} a $k$-power if $P=P_1XP_2$, where $P_1$ and $P_2$ are permutations and $X$ is a $k$-power.
A permutation is called {\em $k$-power-free} if it does not contain $k$-powers.
A $2$-power-free permutation is also called {\em square-free}.
Various generalizations of squares and $k$-powers in permutations can be found in \cite{DGR21,DGR21-2}.

A permutation $Q$ of length $n+1$ is an {\em extension} of a permutation $P$ of length $n$ to the right (left) if the prefix (suffix) of $Q$ of length $n$ is order-isomorphic to $P$.
For example, $4956$ is an extension of $132$ to the right.

A permutation $P$ is called {\em $k$-right-crucial} ({\em $k$-left-crucial}) if $P$ is $k$-power-free and any extension of $P$ 
to the right (left) contains a $k$-power. 
For example, the permutation $3256417$ is 2-right-crucial.
A permutation is called {\em $k$-bicrucial} if it is both $k$-left-crucial and $k$-right-crucial.
Other types of crucial and bicrucial permutations can be found in \cite{AKV12,AKT23,C24}.

The {\em reverse} of a permutation $P=p_1p_2\ldots p_n$ is the permutation $\mathrm{rev}(P)=p_np_{n-1}\ldots p_1$.

Let $k$ and $r$ be integers. Given a permutation $P=p_1\ldots p_n$, define a permutation $k\cdot P+r$ by 
$k\cdot P+r=(kp_1+r)\ldots (kp_n+r)$.

\section{Preliminaries}\label{Sec:Prelim}

\textbf{Levels in permutations.}
Let $P=p_1\ldots p_n$ be a permutation and let $P$ contain no squares of length $4$.
Then there exists $i\in \{0,1,2,3\}$ such that for every non-negative integer $t$ and for all valid indices the inequalities

\begin{equation}\label{Eq:1}
p_{i+4t}<p_{i+4t\pm 1} \textrm{ and } p_{i+4t+2}>p_{i+4t+2\pm 1}
\end{equation}
hold (see \cite{AKPV11}). 
For a permutation $P$ satisfying (\ref{Eq:1}), 
we say that the symbols of $P$ with the indices $4t+i$ form the {\em lower} level,
the symbols of $P$ with the indices $4t+i\pm 1$ form the {\em medium} level, and
the symbols of $P$ with the indices $4t+i+2$ form the {\em upper} level.
For example, the permutation $2463157$ satisfies (\ref{Eq:1}) for $i=1$.
The symbols $2$ and $1$ form its lower level, the symbols $6$ and $7$ form its upper level, and the symbols $4$, $3$ and $5$ form its medium level.

\textbf{High-medium-low construction.} 
Now we briefly recall one construction of square-free permutations proposed by Avgustinovich et al. in \cite{AKPV11}.
Suppose $P$ is a square-free permutation.
Let $Q$ be a permutation satisfying the following conditions:

(S1) $Q$ satisfies (\ref{Eq:1}) for some $i\in \{0,1,2,3\}$,

(S2) any symbol from the lower level is less than any symbol from the medium level and any symbol from the medium level is less than any symbol from the upper level,  

(S3) symbols from the medium level form a permutation that is order-isomorphic to $P$.

Then $Q$ is square-free.

In what follows, we will need the following fact.
\begin{proposition}\label{Prop:2341}
Let $P$ be a permutation constructed by the high-medium-low construction. 
Then any factor of $P$ of length $4$ is not order-isomorphic to $2341$ and $3214$.
\end{proposition}
\begin{proof}
Let $V=xyzt$ be a factor of $P$ of length $4$ and let $V\sim 2341$.
Since $x<y<z$, $x$, $y$ and $z$ belong to the lower, medium and upper levels respectively.
This implies that $t$ belongs to the medium level. Therefore, $t>x$, i.e. $V\not\sim 2341$.
The proof for the case $V\sim 3214$ is similar.
\end{proof}

We also need the following two simple observations (both follow directly from the definitions).

\begin{proposition}\label{Prop:Crucial}
The following statements hold:
 
 \begin{enumerate}
 
 \item Let $P$ and $Q$ be order-isomorphic permutations. If $P$ is $k$-right-crucial, then $Q$ is also $k$-right-crucial.
 
 \item Let $P$ be a permutation. If $P$ is $k$-right-crucial, then $\mathrm{rev}(P)$ is $k$-left-crucial.  
 
 \end{enumerate}
\end{proposition}

\begin{proposition}\label{Prop:Extension}
Let $V$ be a permutation. The following statements hold:
 
 \begin{enumerate}
 
 \item Let $S$ be a suffix of $V$. If any extension of $S$ to the right contains a $k$-power, then any extension of $V$ to the right also contains a $k$-power.
 
 \item Let $P$ be a prefix of $V$. If any extension of $P$ to the left contains a $k$-power, then any extension of $V$ to the left also contains a $k$-power.  
 
 \end{enumerate}
\end{proposition}

\section{Constructions of right-crucial $k$-power-free permutations}\label{Sec:ConstrCrucial}

In this section, we give constructions of $k$-right-crucial permutations.
Firstly, in Subsection \ref{SubSec:Short}, we present constructions $R_k$ of $k$-right-crucial permutations of length $(k-1)(2k+1)$.
Then, in Subsection \ref{SubSec:Long}, we present constructions $R_{m,k}$ of $k$-right-crucial permutations of length $m+(k-1)(2k+1)$ for all $m\geq 1$.

\subsection{Short right-crucial $k$-power-free permutations}\label{SubSec:Short}
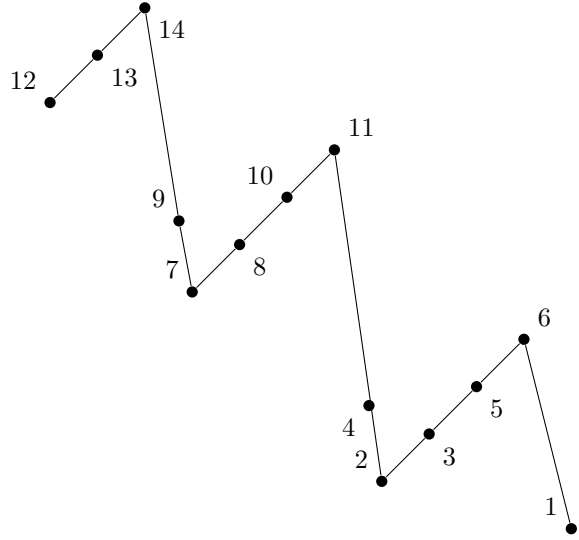
\begin{figure}[h!]
\centering
\begin{tikzpicture}[scale=0.63]
    \node (A1) at (14,3) [circle,fill,inner sep=1.5pt,label=above left:1] {};
    \node (A2) at (10,4) [circle,fill,inner sep=1.5pt,label=above left:2] {};
    \node (A3) at (11,5) [circle,fill,inner sep=1.5pt,label=below right:3] {};
    \node (A4) at (9.73,5.6) [circle,fill,inner sep=1.5pt,label=below left:4] {};
    \node (A5) at (12,6) [circle,fill,inner sep=1.5pt,label=below right:5] {};
    \node (A6) at (13,7) [circle,fill,inner sep=1.5pt,label=above right:6] {};
    \node (A7) at (6,8) [circle,fill,inner sep=1.5pt,label=above left:7] {};
    \node (A8) at (7,9) [circle,fill,inner sep=1.5pt,label=below right:8] {};
    \node (A9) at (5.72,9.5) [circle,fill,inner sep=1.5pt,label=above left:9] {};
    \node (A10) at (8,10) [circle,fill,inner sep=1.5pt,label=above left:10] {};
    \node (A11) at (9,11) [circle,fill,inner sep=1.5pt,label=above right:11] {};
    \node (A12) at (3,12) [circle,fill,inner sep=1.5pt,label=above left:12] {};
    \node (A13) at (4,13) [circle,fill,inner sep=1.5pt,label=below right:13] {};
    \node (A14) at (5,14) [circle,fill,inner sep=1.5pt,label=below right:14] {};
    
    \draw (A1) -- (A6);
    \draw (A2) -- (A3);
    \draw (A3) -- (A5);
    \draw (A5) -- (A6);
    \draw (A10) -- (A11);
    \draw (A2) -- (A11);
    \draw (A7) -- (A8);
    \draw (A8) -- (A10);    
    \draw (A7) -- (A9);
    \draw (A9) -- (A14);
    \draw (A12) -- (A13);
    \draw (A13) -- (A14);
\end{tikzpicture}
\caption{A $3$-right-crucial permutation $R_3$ of length $14$.}
\label{fig:R_3}
\end{figure}

For $k\geq 3$, we define permutations $T_k$ and $N_k$ of lengths $2k-3$ and $2k-1$ by
$T_k=(2k^2-3k+3)(2k^2-3k+4)\ldots (2k^2-k-1)$  and $N_k=31245\ldots (2k-1)$, where 
$5\ldots (2k-1)$ is an increasing sequence of consecutive integers.
For $k\geq 3$ and $i\in [k-1]$, denote $N_{k,i}=N_k+1+(i-1)(2k-1)$.
For $k\geq 3$, we define a permutation $R_k$ of length $(k-1)(2k+1)$ as follows:
$$R_k=T_kN_{k,k-1}N_{k,k-2}\ldots N_{k,1}1$$

\begin{example}\label{Ex:R_3}
For $k=3$ we have $T_3=(12)(13)(14)$, $N_3=31245$, $N_{3,2}=978(10)(11)$, $N_{3,1}=42356$ and 
$R_3=(12)(13)(14)978(10)(11)423561$ (see Figure \ref{fig:R_3}).
\end{example}

\begin{lemma}\label{L:R_k}
The permutation $R_k$ is $k$-right-crucial for all $k\geq 3$. 
\end{lemma}
\begin{proof}
Firstly, we show that any extension of $R_k$ to the right contains a $k$-power.
Suppose that an extension of $R_k$ to the right ends with the symbols $x$ and $y$.
If $x<y$, then its last $2k$ symbols form a $k$-power $X_1\ldots X_k$, 
where $X_i\sim 12$ for all $i\in [k]$.
If $x>y$, then this extension is a $k$-power $Y_1\ldots Y_k$, 
where $Y_i\sim 345\ldots (2k-1)21$ for all $i\in [k]$.

So, it remains to prove that $R_k$ is $k$-power-free.
Analyzing the structure of the permutation $R_k$, we obtain the following:
\begin{itemize}
    \item $R_k$ contains exactly $k-1$ factors that are order-isomorphic to $321$.

    \item $R_k$ contains exactly $k-1$ factors that are order-isomorphic to $312$.

    \item $R_k$ contains exactly $k$ factors that are order-isomorphic to $231$.

    \item $R_k$ does not contain factors that are order-isomorphic to $132$ and $213$.
\end{itemize}
Suppose $R_k$ contains a $k$-power $X=X_1\ldots X_k$.
Let us show that $X_j$ does not contain factors that are order-isomorphic to permutations from the set $\{321,312,231\}$ for any $j\in [k]$.
Let us consider two cases.

Suppose $X_j$ contains a factor that is order-isomorphic to $321$ ($312$) for some $j\in [k]$.
Since $X$ is a $k$-power, $X$ contains at least $k$ factors that are order-isomorphic to $321$ ($312$).
This contradicts the fact that $R_k$ contains exactly $k-1$ factors that are order-isomorphic to $321$ ($312$).

Suppose $X_j$ contains a factor that is order-isomorphic to $231$ for some $j\in [k]$.
Since $X$ is a $k$-power, $X$ contains at least $k$ factors that are order-isomorphic to $231$.
On the other hand, $R_k$ contains exactly $k$ factors that are order-isomorphic to $231$.
Consequently, $X_i$ contains exactly one factor that is order-isomorphic to $231$ for any $i\in [k]$
(we denote such a factor by $U_i$).
Note that $|X_1|=3+d(U_1,U_2)$, where $d(U_1,U_2)$ is the number of symbols of $R_k$ between $U_1$ and $U_2$.
In addition, we have $d(U_1,U_2)\geq 2k-4$.
Therefore, $|X_1|\geq 2k-1$ and $|X|=k|X_1|\geq 2k^2-k$.
Then $|X|>|R_k|$ and we get a contradiction.

As we proved above, any factor of $X_j$ of length $3$ is not order-isomorphic to $321$, $312$ and $231$ for all $j\in [k]$.
In addition, $R_k$ does not contain factors that are order-isomorphic to $132$ and $213$.
Therefore, any factor of $X_j$ of length $3$ is order-isomorphic to $123$ for all $j\in [k]$. 
So, for every $j\in [k]$ the symbols of $X_j$ form an increasing sequence.
This implies that $X_j$ does not contain the first symbols of the permutations $N_{k,k-1},N_{k,k-2},\ldots,N_{k,1}$ for all $j\in [k]$
(otherwise $X_j$ contains two consecutive symbols $x$ and $y$ such that $x>y$).
The same arguments show that $X_j$ does not contain the symbol $1$ for all $j\in [k]$.
Hence, $X$ also does not contain the first symbols of the permutations $N_{k,k-1},N_{k,k-2},\ldots,N_{k,1}$ and 
$X$ does not contain the symbol $1$. 
Therefore, $X$ is a factor of $T_k$ or $X$ is a factor $N_{k,i}$ for some  $i\in [k-1]$.
In both cases, we have $|X|\leq 2k-1$, i.e. $X$ is not a $k$-power. 
\end{proof}

The following two properties of factors of $R_k$ directly follow from the definition of $R_k$.
\begin{lemma}\label{L:FactorsR_k}
Let $k\geq 3$. The following statements hold:
\begin{enumerate}
  \item Any factor of $R_k$ of length $4$ is not order-isomorphic to $1243$.
  
  \item Any factor of $R_k$ of length $3$ is not order-isomorphic to $213$.
\end{enumerate}
\end{lemma}

\subsection{Long right-crucial $k$-power-free permutations}\label{SubSec:Long}
For $m\geq 1$ and $k\geq 3$, we define a permutation $T_{m,k}$ of length $2k-3$ by 
$$T_{m,k}=(2k^2-3k+3)(2k^2-3k+4+m)(2k^2-3k+5+m)\ldots (2k^2-k-1+m).$$
For $m\geq 1$ and $k\geq 3$, we define a permutation $R_{m,k}$ of length $(k-1)(2k+1)$ as follows:
$$R_{m,k}=T_{m,k}N_{k,k-1}N_{k,k-2}\ldots N_{k,1}1$$

\begin{remark}\label{Remark:1}
Note that $R_{m,k}\sim R_k$.
\end{remark}

\begin{example}\label{Ex:R_{4,3}}
For $m=4$ and $k=3$ we have $T_{4,3}=(12)(17)(18)$, $N_3=31245$, $N_{3,2}=978(10)(11)$, $N_{3,1}=42356$ and 
$R_{4,3}=(12)(17)(18)978(10)(11)423561$.
\end{example}

Let $m\geq 1$ and $k\geq 3$.
Using the high-medium-low construction, we construct a square-free permutation $F_{m,k}$ of length $m+1$ satisfying the following conditions:

(A1) the symbols of $F_{m,k}$ are the numbers from the set $\{2k^2-3k+3,2k^2-3k+4,\ldots,2k^2-3k+3+m\}$,

(A2) the last symbol of $F_{m,k}$ belongs to the lower level,

(A3) the symbols from the upper level form an increasing sequence,

(A4) the symbols from the lower level form a decreasing sequence,

(A5) if $y_1$ and $y_2$ are the penultimate and last symbols of the medium level, then $y_1<y_2$.

\begin{remark}
There are many square-free permutations satisfying (A1)--(A5). 
We choose an arbitrary one of them. 
For example, the permutation $(14)(13)(15)(17)(16)(12)$ satisfies (A1)--(A5) for $m=5$ and $k=3$.

\end{remark}

The prefix of $F_{m,k}$ of length $m$ is denoted by $F'_{m,k}$.
In what follows, we will need the following properties of factors of $R_{m,k}$ and $F'_{m,k}$. 

\begin{lemma}\label{L:BasicPrR_m,k}
 Let $m\geq 1$ and $k\geq 3$. The following statements hold:
\begin{enumerate}
 
  \item Any factor of $R_{m,k}$ of length $4$ is not order-isomorphic to $1243$.
  
  \item Any factor of $R_{m,k}$ of length $3$ is not order-isomorphic to $213$.
  
  \item The permutation $R_{m,k}$ is $k$-right-crucial.
  
  \item The permutation $F'_{m,k}$ is square-free.

\end{enumerate}
\end{lemma}
\begin{proof}

1, 2. It follows from Remark \ref{Remark:1} and Lemma \ref{L:FactorsR_k}.

3. It follows from Remark \ref{Remark:1}, Lemma \ref{L:R_k} and Proposition \ref{Prop:Crucial}.

4. Since $F_{m,k}$ is square-free and $F'_{m,k}$ is a factor of $F_{m,k}$, $F'_{m,k}$ is also square-free.
\end{proof}

Finally, for $m\geq 1$ and $k\geq 3$, we define a permutation $\widetilde{R}_{m,k}$ of length $m+(k-1)(2k+1)$ by
$\widetilde{R}_{m,k}=F'_{m,k}R_{m,k}$.
In Section \ref{Sec:Crucial}, we will show that $\widetilde{R}_{m,k}$ is $k$-right-crucial.

\section{Constructions of bicrucial $k$-power-free permutations}\label{Sec:ConstrBicrucial}
In this section, we give constructions of $k$-bicrucial permutations.

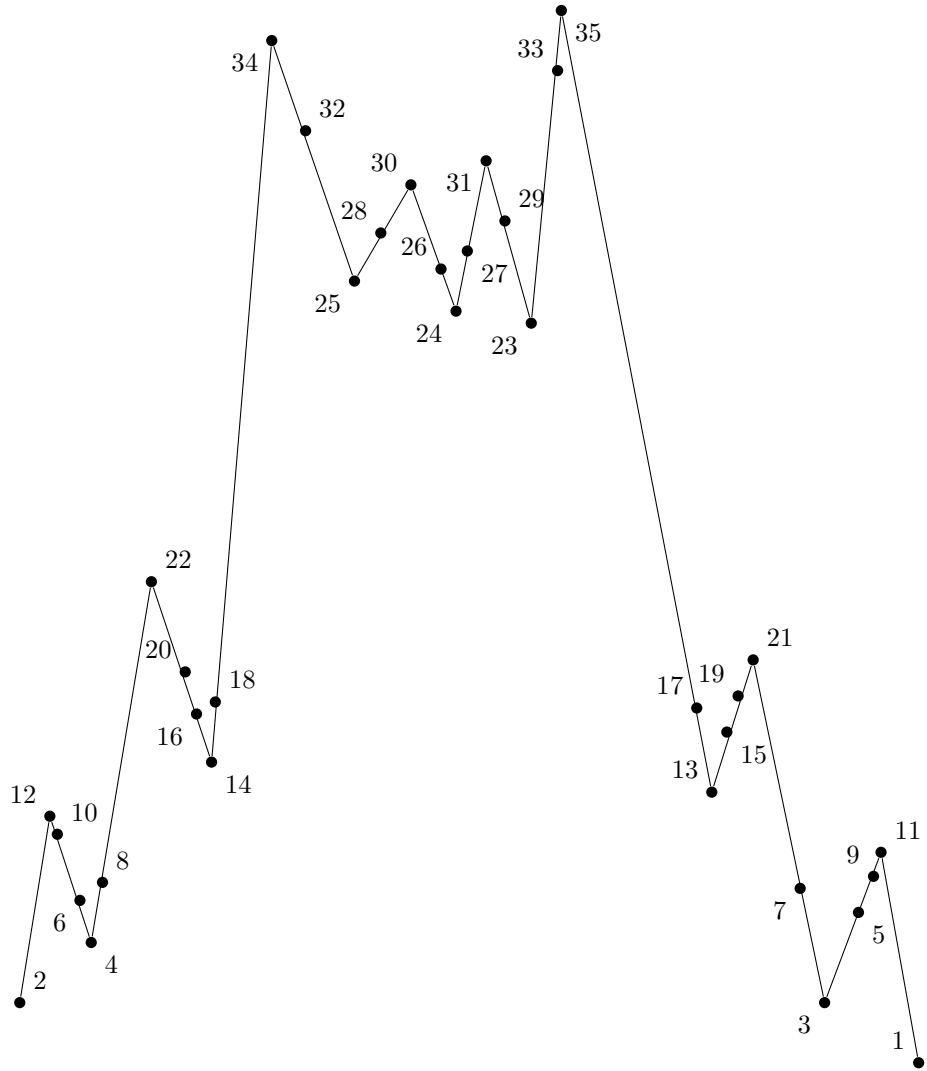
\begin{figure} 
\begin{tikzpicture}[xscale=0.5, yscale=0.8] %
\node (A1) at (23.5,1.5) [circle,fill,inner sep=1.5pt,label=above left:1,label distance=2mm] {};
\node (A2) at (-0.4,2.5) [circle,fill,inner sep=1.5pt,label=above right:2,label distance=2mm] {};
\node (A3) at (21,2.5) [circle,fill,inner sep=1.5pt,label=below left:3,label distance=2mm] {};
\node (A4) at (1.5,3.5) [circle,fill,inner sep=1.5pt,label=below right:4,label distance=2mm] {};
\node (A5) at (21.9,4.0) [circle,fill,inner sep=1.5pt,label=below right:5,label distance=2mm] {};
\node (A6) at (1.2,4.2) [circle,fill,inner sep=1.5pt,label=below left:6,label distance=2mm] {};
\node (A7) at (20.35,4.4) [circle,fill,inner sep=1.5pt,label=below left:7,label distance=2mm] {};
\node (A8) at (1.8,4.5) [circle,fill,inner sep=1.5pt,label=above right:8,label distance=2mm] {};
\node (A9) at (22.3,4.6) [circle,fill,inner sep=1.5pt,label=above left:9,label distance=3mm] {};
\node (A10) at (0.60,5.3) [circle,fill,inner sep=1.5pt,label=above right:10,label distance=2mm] {};
\node (A11) at (22.5,5) [circle,fill,inner sep=1.5pt,label=above right:11,label distance=3mm] {};
\node (A12) at (0.4, 5.6) [circle,fill,inner sep=1.5pt,label=above left:12,label distance=2mm] {};
\node (A13) at (18,6) [circle,fill,inner sep=1.5pt,label=above left:13,label distance=2mm] {};
\node (A14) at (4.7,6.5) [circle,fill,inner sep=1.5pt,label=below right:14,label distance=2mm] {};
\node (A15) at (18.4,7) [circle,fill,inner sep=1.5pt,label=below right:15,label distance=3mm] {};
\node (A16) at (4.3,7.3) [circle,fill,inner sep=1.5pt,label=below left:16,label distance=2mm] {};
\node (A17) at (17.6,7.4) [circle,fill,inner sep=1.5pt,label=above left:17,label distance=2mm] {};
\node (A18) at (4.8,7.5) [circle,fill,inner sep=1.5pt,label=above right:18,label distance=2mm] {};
\node (A19) at (18.7,7.6) [circle,fill,inner sep=1.5pt,label=above left:19,label distance=3mm] {};
\node (A20) at (4.0,8.0) [circle,fill,inner sep=1.5pt,label=above left:20,label distance=2mm] {};
\node (A21) at (19.1,8.2) [circle,fill,inner sep=1.5pt,label=above right:21,label distance=3mm] {};
\node (A22) at (3.1,9.5) [circle,fill,inner sep=1.5pt,label=above right:22,label distance=2mm] {};
\node (A23) at (13.2,13.8) [circle,fill,inner sep=1.5pt,label=below left:23,label distance=3mm] {};
\node (A24) at (11.2,14) [circle,fill,inner sep=1.5pt,label=below left:24,label distance=3mm] {};
\node (A25) at (8.5,14.5) [circle,fill,inner sep=1.5pt,label=below left:25,label distance=3mm] {};
\node (A26) at (10.8,14.7) [circle,fill,inner sep=1.5pt,label=above left:26,label distance=2mm] {};
\node (A27) at (11.5,15) [circle,fill,inner sep=1.5pt,label=below right:27,label distance=2mm] {};
\node (A28) at (9.2,15.3) [circle,fill,inner sep=1.5pt,label=above left:28,label distance=2mm] {};
\node (A29) at (12.5,15.5) [circle,fill,inner sep=1.5pt,label=above right:29,label distance=2mm] {};
\node (A30) at (10,16.1) [circle,fill,inner sep=1.5pt,label=above left:30,label distance=3mm] {};
\node (A31) at (12,16.5) [circle,fill,inner sep=1.5pt,label=below left:31,label distance=3mm] {};
\node (A32) at (7.2,17) [circle,fill,inner sep=1.5pt,label=above right:32,label distance=2mm] {};
\node (A33) at (13.9,18) [circle,fill,inner sep=1.5pt,label=above left:33,label distance=2mm] {};
\node (A34) at (6.3,18.5) [circle,fill,inner sep=1.5pt,label=below left:34,label distance=3mm] {};
\node (A35) at (14,19) [circle,fill,inner sep=1.5pt,label=below right:35,label distance=3mm] {};

\draw (A2) -- (A12);
\draw (A12) -- (A4);
\draw (A4) -- (A22);
\draw (A22) -- (A14);
\draw (A14) -- (A34);
\draw (A34) -- (A25);
\draw (A25) -- (A30);
\draw (A30) -- (A24);
\draw (A24) -- (A31);
\draw (A31) -- (A23);
\draw (A23) -- (A35);
\draw (A35) -- (A13);
\draw (A13) -- (A21);
\draw (A21) -- (A3);
\draw (A3) -- (A11);
\draw (A11) -- (A1);
\end{tikzpicture}
\caption{A 3-bicrucial permutation $B_{1,3}$ of length $35$ corresponding to the permutation 
$H_{1,3}=(25)(28)(30)(26)(24)(27)(31)(29)(23)$ from Remark \ref{Rem:H_1,3}.}
\label{fig:B_{m,k}}
\end{figure}

For $m\geq 1$ and $k\geq 3$, we define permutations $Q_{m,k}$ and $W_{m,k}$ of length $2k-4$ as follows:
$$Q_{m,k}=(4k^2-6k+8m+4k-4)(4k^2-6k+8m+4k-6)\ldots (4k^2-6k+8m+6)$$ and

$$W_{m,k}=(4k^2-6k+8m+7)(4k^2-6k+8m+9)\ldots (4k^2-6k+8m+4k-3).$$

For $m\geq 1$ and $k\geq 3$, we define permutations $P_{m,k}$ and $S_{m,k}$ of length $(k-1)(2k+1)$ as follows:
$$P_{m,k}=2(2\cdot\mathrm{rev}(N_{k,1}))(2\cdot\mathrm{rev}(N_{k,2}))\ldots (2\cdot\mathrm{rev}(N_{k,k-1}))Q_{m,k}(4k^2-6k+5+2m)$$ and

$$S_{m,k}=(4k^2-6k+5)W_{m,k}(2\cdot N_{k,k-1}-1)(2\cdot N_{k,k-2}-1)\ldots (2\cdot N_{k,1}-1)1$$

\begin{remark}\label{Remark:S_m,k}
Note that $S_{m,k}\sim R_k$ and $P_{m,k}\sim \mathrm{rev}(R_k)$.
\end{remark}

\begin{example}
For $m=1$ and $k=3$ we have $Q_{1,3}=(34)(32)$, $W_{1,3}=(33)(35)$, $P_{1,3}=2(12)(10)648(22)(20)(16)(14)(18)(34)(32)(25)$ and 
$S_{1,3}=(23)(33)(35)(17)(13)(15)(19)(21)7359(11)1$ (see Figure \ref{fig:B_{m,k}}).
\end{example}

Let $m\geq 1$ and $k\geq 3$.
Using the high-medium-low construction, we construct a square-free permutation $H_{m,k}$ of length $8m+1$ satisfying the following conditions:

(B1) the symbols of $H_{m,k}$ are the numbers from the set $\{4k^2-6k+5,4k^2-6k+6,\ldots,4k^2-6k+5+8m\}$, 
  
  
(B2) the first and last symbols of $H_{m,k}$ belong to the lower level,
  
(B3) the symbols from the upper level form an increasing sequence,

(B4) the symbols from the lower level form a decreasing sequence,
  

(B5) if $y_1$ and $y_2$ are the penultimate and last symbols of the medium level, then $y_1<y_2$.

\begin{remark}\label{Rem:H_1,3}
There are many square-free permutations satisfying (B1)--(B5). 
We choose an arbitrary one of them. 
For example, the permutation $(25)(28)(30)(26)(24)(27)(31)(29)(23)$ satisfies conditions (B1)--(B5) for $m=1$ and $k=3$ (see Figure \ref{fig:B_{m,k}}).
\end{remark}

\begin{remark}
Since $|H_{m,k}|=8m+1$ and $H_{m,k}$ ends on the lower level, the medium level of $H_{m,k}$ consists of $4m$ symbols.
We denote these symbols by $x_1,x_2,\ldots,x_{4m}$. We claim that $x_1>x_2$.
Taking into account (\ref{Eq:1}) and (B5), we get that $4m=i+4t+1$ or $4m=i+4t+2$. Therefore, $i=3$ or $i=2$.
In both cases, we have $x_1>x_2$ due to (\ref{Eq:1}).
\end{remark}

Let $H''_{m,k}$ be the permutation obtained from $H_{m,k}$ by deleting the first and last symbols.
In what follows, we will need the following result.

\begin{lemma}\label{L:BasicPrBicrucial}
 Let $m\geq 1$ and $k\geq 3$. The following statements hold:
\begin{enumerate}
 
  \item Any factor of $S_{m,k}$ of length $4$ is not order-isomorphic to $1243$.
  
  \item The permutation $S_{m,k}$ is $k$-right-crucial.
  
  \item The permutation $P_{m,k}$ is $k$-left-crucial.
  
  
   \item The permutation $H''_{m,k}$ is square-free.

   \item Any factor of $H''_{m,k}$ of length $4$ is not order-isomorphic to $2341$ and $3214$.

\end{enumerate}
\end{lemma}
\begin{proof}
1. It follows from Remark \ref{Remark:S_m,k} and Lemma \ref{L:FactorsR_k}.

2, 3. It follows from Remark \ref{Remark:S_m,k}, Lemma \ref{L:R_k} and Proposition \ref{Prop:Crucial}.

4. Since $H_{m,k}$ is square-free and $H''_{m,k}$ is a factor of $H_{m,k}$, $H''_{m,k}$ is also square-free.

5. It follows from Proposition \ref{Prop:2341}.
\end{proof}

Finally, for $m\geq 1$ and $k\geq 3$, we define a permutation $B_{m,k}$ of length $8m-1+2(k-1)(2k+1)$ by
$B_{m,k}=P_{m,k}H''_{m,k}S_{m,k}$.
In Section \ref{Sec:Bicrucial}, we will show that $B_{m,k}$ is $k$-bicrucial.

\section{Right-crucial $k$-power-free permutations}\label{Sec:Crucial}
In this section, we prove that for every $k\geq 3$ there exist $k$-right-crucial permutations of any length at least $(k-1)(2k+1)$.
Firstly, we prove the following result.

\begin{lemma}\label{L:V1V2-1}
Let $V_1V_2$ be a factor of $\widetilde{R}_{m,k}$, where $|V_1|=|V_2|\geq 2$.
If $V_1$ contains the last symbol of $F'_{m,k}$, then $V_1$ and $V_2$ are not order-isomorphic.
\end{lemma}
\begin{proof}
Let $y$ be the last symbol of $F'_{m,k}$.
Let $V'_1$ be the prefix of $V_1$ ending with the symbol $y$, 
and let $V''_1$ be the suffix of $V_1$ obtained from $V_1$ by removing the prefix $V'_1$.

Suppose that $|V'_1|\geq 4$.
The suffix of $V'_1$ of length $4$ is order-isomorphic to $1243$.
On the other hand, $V_2$ is a factor $R_{m,k}$. Then, by Lemma \ref{L:BasicPrR_m,k}, any factor of $V_2$ of length $4$ is not order-isomorphic to $1243$.
Consequently, in this case $V_1\not\sim V_2$.

Suppose that $|V''_1|\geq 2$. Let $U$ be the factor of $V_1$ consisting of $y$ and the first two symbols of $V''_1$.
Note that $U\sim 213$. On the other hand, $V_2$ is a factor $R_{m,k}$. Then, by Lemma \ref{L:BasicPrR_m,k}, any factor of $V_2$ of length $3$ is not order-isomorphic to $213$.
Consequently, in this case $V_1\not\sim V_2$.

So, we can assume that $|V'_1|\leq 3$ and $|V''_1|\leq 1$. Therefore, it remains to consider the following five cases.

\textbf{Case $1$: $|V'_1|=3$ and $|V''_1|=1$.} 
In this case, we have $V_1\sim 2431$. On the other hand, $V_2\sim 1234$ for $k>3$ and $V_2\sim 3421$ for $k=3$.
Hence, $V_1\not\sim V_2$.

\textbf{Case $2$: $|V'_1|=3$ and $|V''_1|=0$.} 
In this case, $V_1\sim 132$ and $V_2\sim 123$, that is, $V_1\not\sim V_2$.

\textbf{Case $3$: $|V'_1|=2$ and $|V''_1|=1$.} 
In this case, we have $V_1\sim 321$. On the other hand, $V_2\sim 123$ for $k>3$ and $V_2\sim 231$ for $k=3$.
Hence, $V_1\not\sim V_2$.

\textbf{Case $4$: $|V'_1|=2$ and $|V''_1|=0$.} 
In this case, $V_1\sim 21$ and $V_2\sim 12$, that is, $V_1\not\sim V_2$.

\textbf{Case $5$: $|V'_1|=1$ and $|V''_1|=1$.} 
In this case, $V_1\sim 21$ and $V_2\sim 12$, that is, $V_1\not\sim V_2$.

\end{proof}

\begin{remark}\label{Remark:V1V2}
Note that in the statement of Lemma \ref{L:V1V2-1} we can replace $\widetilde{R}_{m,k}$ with $H''_{m,k}S_{m,k}$ and $F'_{m,k}$ with $H''_{m,k}$.  
Namely, the following statement holds:
\begin{itemize}
  \item Let $V_1V_2$ be a factor of $H''_{m,k}S_{m,k}$, where $|V_1|=|V_2|\geq 2$.
If $V_1$ contains the last symbol of $H''_{m,k}$, then $V_1$ and $V_2$ are not order-isomorphic.
\end{itemize}
The proof of this fact is similar to the proof of Lemma \ref{L:V1V2-1}.
\end{remark}

The main result of this section is the following.

\begin{theorem}\label{Th:1}
The permutation $\widetilde{R}_{m,k}$ is $k$-right-crucial for all $m\geq 1$ and $k\geq 3$.
\end{theorem}
\begin{proof}
Lemma \ref{L:BasicPrR_m,k} implies that any extension of $R_{m,k}$ to the right contains a $k$-power.
In addition, $R_{m,k}$ is a suffix of $\widetilde{R}_{m,k}$. 
Therefore, any extension of $\widetilde{R}_{m,k}$ to the right also contains a $k$-power due to Proposition \ref{Prop:Extension}.

So, it remains to prove that $\widetilde{R}_{m,k}$ is $k$-power-free.
Suppose $\widetilde{R}_{m,k}$ contains a $k$-power $X=X_1\ldots X_k$.
Let $y$ be the last symbol of $F'_{m,k}$.
By Lemma \ref{L:BasicPrR_m,k}, $F'_{m,k}$ is square-free and $R_{m,k}$ is $k$-power-free.
This implies that $X$ contains the symbol $y$.
Therefore, $X_j$ contains $y$ for some $j\in [k]$.
Let us show that $j=k$. If $j<k$, then $X_{j}X_{j+1}$ is a factor of $\widetilde{R}_{m,k}$.
Applying Lemma \ref{L:V1V2-1} for $V_1=X_j$ and $V_2=X_{j+1}$, we obtain that $X_j$ and $X_{j+1}$ are not order-isomorphic.
Consequently, in this case $X$ is not a $k$-power. So, we have $j=k$.
Since $j=k$ and $k\geq 3$, $X_{k-2}X_{k-1}$ is a factor of $F'_{m,k}$.
This contradicts the fact that $F'_{m,k}$ is square-free.
Thus, $\widetilde{R}_{m,k}$ is $k$-power-free.

\end{proof}

\begin{corollary}
For every $k\geq 3$, there exist $k$-right-crucial permutations of any length at least $(k-1)(2k+1)$.
\end{corollary}

\begin{remark}\label{Remark:P_m,kH_m,kS_m,k}
According to Remark \ref{Remark:V1V2}, the proof of Theorem \ref{Th:1} also works for the permutation $H''_{m,k}S_{m,k}$.
So, we have the following:
\begin{itemize}
  \item The permutation $H''_{m,k}S_{m,k}$ is $k$-right-crucial for all $m\geq 1$ and $k\geq 3$.

\end{itemize}

This result immediately implies the following:

\begin{itemize}
  \item The permutation $P_{m,k}H''_{m,k}$ is $k$-left-crucial for all $m\geq 1$ and $k\geq 3$.
\end{itemize}

\end{remark}

\section{Bicrucial $k$-power-free permutations}\label{Sec:Bicrucial}
In this section, we prove that for every $k\geq 3$ there exist arbitrarily long $k$-bicrucial permutations.
Firstly, we prove two technical lemmas.

\begin{lemma}\label{L:V1V2-2}
Let $V_1V_2$ be a factor of $H''_{m,3}S_{m,3}$, where $|V_1|=|V_2|\geq 2$,
and let $V_2$ contain the last symbol of $H''_{m,3}$. If $V_1$ and $V_2$ are order-isomorphic, 
then $|V_2|=4$ and $V_2$ starts with the last symbol of $H''_{m,3}$.
\end{lemma}
\begin{proof}
Let $y$ be the last symbol of $H''_{m,3}$.
Let $V'_2$ be the prefix of $V_2$ ending with the symbol $y$, 
and let $V''_2$ be the suffix of $V_2$ obtained from $V_2$ by removing the prefix $V'_2$.
Since $H_{m,3}$ is square-free and $V_1\sim V_2$, we have $|V''_2|\geq 2$.

Suppose that $|V''_2|\geq 4$.
The prefix of $V''_2$ of length $4$ is order-isomorphic to $2341$.
On the other hand, $V_1$ is a factor $H''_{m,3}$. Then, by Lemma \ref{L:BasicPrBicrucial}, any factor of $V_1$ of length $4$ is not order-isomorphic to $2341$.
Consequently, in this case $V_1\not\sim V_2$.

Suppose that $|V'_2|\geq 2$.
As we noted above, we can assume that $|V''_2|\geq 2$.
Let $V$ be the factor of $V_2$ consisting of the last two symbols of $V'_2$ and the first two symbols of $V''_2$.
Note that $V\sim 3214$. On the other hand, $V_1$ is a factor $H''_{m,3}$. Then, by Lemma \ref{L:BasicPrBicrucial}, any factor of $V_1$ of length $4$ is not order-isomorphic to $3214$.
Consequently, in this case $V_1\not\sim V_2$.

Suppose that $|V'_2|=1$ and $|V''_2|=2$.
In this case, we have $V_2\sim 213$ and $V_1\sim 123$, that is, $V_1\not\sim V_2$.

Therefore, if $V_1\sim V_2$, then $|V'_2|=1$ and $|V''_2|=3$. Thus,  $|V_2|=4$ and $V_2$ starts with $y$.
\end{proof}

\begin{lemma}\label{L:V1V2-3}
Let $V_1V_2$ be a factor of $B_{m,k}$, where $|V_1|=|V_2|$.
If $V_i$ contains the first and last symbols of $H''_{m,k}$ for some $i\in \{1,2\}$, then $V_1$ and $V_2$ are not order-isomorphic.
\end{lemma}
\begin{proof}
Assume that $i=1$.
Let $U$ be the suffix of $H''_{m,k}$ of length $4$.
Note that $U\sim 1243$. Since $V_1$ contains the first and last symbols of $H''_{m,k}$, $U$ is a factor of $V_1$.
On the other hand, $V_2$ is a factor $S_{m,k}$. Then, by Lemma \ref{L:BasicPrBicrucial}, any factor of $V_2$ of length $4$ is not order-isomorphic to $1243$.
Therefore,  $V_1$ and $V_2$ are not order-isomorphic.
The proof for $i=2$ is similar.
\end{proof}

The main result of this paper is the following.

\begin{theorem}\label{Th:2}
The permutation $B_{m,k}$ is $k$-bicrucial for all $m\geq 1$ and $k\geq 3$.
\end{theorem}
\begin{proof}
Lemma \ref{L:BasicPrBicrucial} implies that any extension of $S_{m,k}$ ($P_{m,k}$) to the right (left) contains a $k$-power.
In addition, $S_{m,k}$ ($P_{m,k}$) is a suffix (prefix) of $B_{m,k}$. 
Therefore, any extension of $B_{m,k}$ to the right (left) also contains a $k$-power due to Proposition \ref{Prop:Extension}.

So, it remains to prove that $B_{m,k}$ is $k$-power-free.
Suppose $B_{m,k}$ contains a $k$-power $X=X_1\ldots X_k$. 
Let $x$ and $y$ be the first and last symbols of $H''_{m,k}$.
By Remark~\ref{Remark:P_m,kH_m,kS_m,k} and Lemma \ref{L:BasicPrBicrucial}, the permutations $P_{m,k}H''_{m,k}$ and $S_{m,k}$ are $k$-power-free.
This implies that $X$ contains the symbol $y$.
Similarly we can prove that $X$ contains the symbol $x$.
Therefore, $X_i$ contains $x$ for some $i\in [k]$ and  $X_j$ contains $y$ for some $j\in [k]$.
Let us consider two cases.

\textbf{Case $1$: $i=j$.} Assume that $j<k$. In this case, $X_jX_{j+1}$ is a factor of $B_{m,k}$.
Applying Lemma \ref{L:V1V2-3} for $V_1=X_j$ and $V_2=X_{j+1}$, we obtain that $X_j$ and $X_{j+1}$ are not order-isomorphic.
Consequently, in this case $X$ is not a $k$-power.
The proof for $j=k$ is similar (in this case, we apply Lemma \ref{L:V1V2-3} for $V_1=X_{j-1}$ and $V_2=X_j$).

\textbf{Case $2$: $i<j$.} Assume that $j<k$. Then $X_jX_{j+1}$ is a factor of $H''_{m,k}S_{m,k}$.
Applying Remark \ref{Remark:V1V2} for $V_1=X_j$ and $V_2=X_{j+1}$, we obtain that $X_j$ and $X_{j+1}$ are not order-isomorphic.
Hence, in this case $X$ is not a $k$-power. So, we have $j=k$.
Similar arguments imply that $i=1$.
Since $H''_{m,k}$ is square-free due to Lemma \ref{L:BasicPrBicrucial}, we have $k=3$.
Applying Lemma \ref{L:V1V2-2} for $V_1=X_2$ and $V_2=X_3$, we obtain that $|X_3|=4$ and $X_3$ starts with $y$.
Similarly we can show that $X_1$ ends with $x$.
Then $|H''_{m,k}|=|X_2|+2=6$.
On the other hand, we have $|H''_{m,k}|\geq 7$.
Therefore, $X$ is not a $k$-power.
\end{proof}

\begin{corollary}
For every $k\geq 3$, there exist arbitrarily long $k$-bicrucial permutations.
\end{corollary}

\section{Open problems}\label{Open problem}
We proved that for every $k\geq 3$ there exist right-crucial $k$-power-free permutations of any length at least $(k-1)(2k+1)$.
On the other hand, computational experiments show that there are no right-crucial $3$-power-free permutations of length less than $14$.
So, it seems interesting to consider the following question.

\begin{problem}
Are there right-crucial $k$-power-free permutations of length less than $(k-1)(2k+1)$ for $k\geq 4$? 
\end{problem}




\begin{thebibliography}{99}

\bibitem{AGHK10}
S. Avgustinovich, A. Glen, B. V. Halld\'{o}rsson, S. Kitaev, On shortest crucial words avoiding abelian powers,
Discrete Applied Mathematics 158(6) (2010) 605--607.

\bibitem{AKPV11}
S. Avgustinovich, S. Kitaev, A. Pyatkin, A. Valyuzhenich, On square-free permutations, Journal of Automata, Languages and Combinatorics 16(1) (2011) 3--10.

\bibitem{AKV12}
S. Avgustinovich, S. Kitaev, A. Valyuzhenich, Crucial and bicrucial permutations with respect to arithmetic monotone patterns, Siberian Electronic Mathematical Reports 9 (2012) 660--671.


\bibitem{AKT23}
S. Avgustinovich, S. Kitaev, A. Taranenko, On five types of crucial permutations with respect to monotone patterns, The Electronic Journal of Combinatorics
30(1) (2023) \#P1.40. 

\bibitem{B04}
E. M. Bullock, Improved bounds on the length of maximal abelian square-free words, The Electronic Journal of Combinatorics 11(1) (2004) \#R17.

\bibitem{CM01}
L. J. Cummings, M. Mays, A one-sided Zimin construction, The Electronic Journal of Combinatorics 8(1) (2001) \#R27.




\bibitem{C24}
Y. Choi, Counting crucial permutations with respect to monotone patterns, Discrete Mathematics 347(4) (2024) 113861.


\bibitem{DGR21}
A. Dudek, J. Grytczuk, A. Ruci\'nski, Variations on twins in permutations, The Electronic Journal of Combinatorics 28(3) (2021) \#P3.19.


\bibitem{DGR21-2}
A. Dudek, J. Grytczuk, A. Ruci\'nski, Tight multiple twins in permutations, Annals of Combinatorics 25 (2021) 1075--1094.

\bibitem{GKKLN15}
I. Gent, S. Kitaev, A. Konovalov, S. Linton, P. Nightingale, S-crucial and bicrucial permutations with respect to squares, 
Journal of Integer Sequences 18 (2015) Article 15.6.5.


\bibitem{GHK10}
 A. Glen, B. V. Halld\'{o}rsson, S. Kitaev, Crucial abelian $k$-power-free words, Discrete Mathematics and Theoretical Computer Science,
 12(5) (2010) 83--96.


\bibitem{GJ22}
C. Groenland, T. Johnston, The lengths for which bicrucial square-free permutations exist, Enumerative Combinatorics and Applications 
2(4) (2022) Article \#S4PP4.

\bibitem{K03}
M. Korn, Maximal abelian square-free words of short length, Journal of Combinatorial Theory, Series A 102(1) (2003) 207--211.


\end{thebibliography}
\end{document}